\documentclass[12pt]{amsart}
\usepackage{amssymb}%
\usepackage{graphicx}%
\usepackage{amsthm}%
\usepackage{amsmath}%
\usepackage{amsfonts} 
\usepackage{latexsym}%
\usepackage{epsfig}%
\usepackage{epic}%
\usepackage{eufrak}%
\usepackage{amscd}%
\usepackage{color}%
\usepackage[linktoc=all,  hyperindex]{hyperref}%
\newtheorem{theorem}{Theorem}[section]
\newtheorem{thm}[theorem]{Theorem}
\newtheorem{lem}[theorem]{Lemma}
\newtheorem{prop}[theorem]{Proposition}
\newtheorem{cor}[theorem]{Corollary}
\newtheorem{defn}[theorem]{Definition}
\newtheorem{rem}[theorem]{Remark}
\newtheorem{ex}[theorem]{Example}

\newcommand{\Spec}{\operatorname{Spec}\,}

\newcommand{\Ass}{\operatorname{Ass}}
\newcommand{\supp}{\operatorname{supp}\,}

\newcommand{\depth}{\operatorname{depth}\,}
\renewcommand{\dim}{\operatorname{dim}\,}

\newcommand{\Min}{\operatorname{Min}\,}

\newcommand{\h}{\operatorname{ht}\,}

\renewcommand{\H}{\operatorname{H}}
\newcommand{\V}{\operatorname{V}}

\newcommand{\G}{\operatorname{G}}

\newcommand{\fm}{\mathfrak{m}}
\newcommand{\fp}{\mathfrak{p}}
\newcommand{\fq}{\mathfrak{q}}

\newcommand{\fn}{\mathfrak{n}}

\newcommand{\excise}[1]{}
\begin{document}
\bibliographystyle{amsplain}

\title{On certain equidimensional polymatroidal ideals}

\author{Somayeh Bandari}
\address{Somayeh Bandari\\ School of Mathematics, Institute for
Research in Fundamental Sciences (IPM) P. O. Box: 19395-5746,
Tehran, Iran.} \email{somayeh.bandari@yahoo.com}

\author{Raheleh Jafari}
\address{Raheleh Jafari\\School of Mathematics, Institute for
Research in Fundamental Sciences (IPM) P. O. Box: 19395-5746,
Tehran, Iran.} \email{rjafari@ipm.ir}


\keywords{Connected in codimension one, equidimensional ideals, generalized Cohen-Macaulay, polymatroidal ideals,  unmixed ideals. \\
Somayeh Bandari was in part supported by a grant from IPM (No.
92130020)\\
Raheleh Jafari was in part supported by a grant from IPM (No.
92130420).}

\maketitle

\begin{abstract}
The class of equidimensional polymatroidal ideals is studied. In
particular, we show that an unmixed polymatroidal ideal is connected
in codimension one if and only if it is Cohen-Macaulay. Especially a
matroidal ideal is connected in codimension one precisely when it is
a squarefree Veronese ideal.  As a consequence we indicate that for
polymatroidal  ideals,  Serre's condition $(S_n)$ for some $n\geq
2$ is equivalent to Cohen-Macaulay property.  We also give a
classification of generalized Cohen-Macaulay polymatroidal ideals.
\end{abstract}

\section{introduction}

\excise{The Cohen-Macaulay property of monomial ideals has been
studied  by several authors, to study arithmetic properties  of
related objects like simplicial complexes, graphs and discrete
polymatroids.

The class of polymatroidal ideals (the monomial ideals corresponding
to the bases of  discrete polymatroids) has been studied }

Throughout this paper we consider monomial ideals of the polynomial
ring $S=k[x_1,\ldots,x_n]$  over a field $k$, and
$\fm=(x_1,\ldots,x_n)$ denotes the unique homogeneous maximal ideal.
The Cohen-Macaulay polymatroidal ideals are classified by Herzog and
Hibi \cite{HH}, into the principal ideals, the Veronese ideals, and
the squarefree Veronese ideals. As  mentioned in \cite{HH}, it is
natural and interesting  to classify all unmixed polymatroidal
ideals. Recall that an ideal $I$ is called unmixed if all prime
ideals in $\Ass(S/I)$ have the same height. If all minimal  prime
ideals of $I$ have the same height, then $I$ is called
equidimensional. Obviously an unmixed ideal $I$ is equidimensional
and the converse holds precisely when $\Min(S/I)=\Ass(S/I)$. In
particular a squarefree monomial ideal is equidimensional  if and
only if it is unmixed.

In this paper we study certain classes of equidimensional
polymatroidal ideals. After giving some preliminary concepts and
results in Section 2, we study the polymatroidal ideals connected in
codimension one, in Section 3. Consider the Zarisky topology on
$\Spec(S/I)$ for a monomial ideal $I$. $\Spec(S/I)$ is a connected
space with this topology.
The ideal $I$ is called connected in
codimension one, if $\Spec(S/I)$ remains connected after removing
closed subsets with codimension bigger than one \cite{Ha}. This
property can be expressed in terms of minimal prime ideals of
$I$, and implies that
  $I$ is equidimensional (see Remark \ref{equ}). From combinatorial point of view, a squarefree monomial ideal is
connected in codimension one, if it is the Stanley-Reisner ideal of
a strongly connected simplicial complex.

   As  mentioned
  in Remark \ref{def}, Cohen-Macaulay ideals are connected in
  codimension one. The aim of this section is to find when the
  converse holds  true for polymatroidal ideals.
Theorem \ref{pc} states that matroidal ideals connected in
codimension one, are precisely squarefree Veronese ideals and thus Cohen-Macaulay. We extend this result to unmixed
polymatroidal ideals in Theorem \ref{gc}, the essential result of
this section. The main consequence of this result is Corollary
\ref{co} which asserts that for polymatroidal
ideals satisfying Serre's condition $(S_n)$ for some
$n\geq 2$ is equivalent to being
Cohen-Macaulay.

The unmixed polymatroidal ideals have also been studied by
Vl$\breve{a}$doiu in \cite{V}. He shows that an ideal of Veronese
type is unmixed if and only if it is Cohen-Macaulay. Our second
target is to find equidimensional polymatroidal ideals which are not
Cohen-Macaulay, in Section 4. We show that a polymatroidal ideal
generated in degree 2, is equidimensional if and only if it is
generalized Cohen-Macaulay (see Proposition \ref{2}) and Example
\ref{exam1}(iii) is a non-Cohen-Macaulay ideal in this class. An
unmixed polymatroidal ideal generated in degree $d>2$, is not
necessarily generalized Cohen-Macaulay (see Example \ref{1.4}). In
the case of matroidal ideals, Theorem \ref{mat} states that
generalized Cohen-Macaulay matroidal ideals generated in degree
$d>2$, are precisely Cohen-Macaulay matroidal ideals.

By \cite[Proposition 5]{HV}, the polymatroidal ideal $I$ generated
in degree $d$ has an irredundant primary decomposition either of the
form $I=J\cap\fm^0$ or $I=J\cap\fm^d$. The classification of
generalized Cohen-Macaulay polymatroidal ideals, stated in Theorem
\ref{th}, indicates that  a fully supported monomial ideal
$I=J\cap\fm^s$ generated in degree $d$ with $s\in\{0,d\}$, is a
generalized Cohen-Macaulay polymatroidal ideal if and only if one of
the following statements holds true:
\begin{enumerate}
\item[a)] $J$ is a Cohen-Macaulay polymatroidal ideal i.e. $J$ is either  a principal ideal, a Veronese ideal, or a squarefree Veronese ideal.
\item[b)] $J=\fp_1^{a_1}\cap\cdots\cap\fp_r^{a_r}$ is equidimensional and $\fp_i+\fp_j=\fm$ for all $i\neq j$.
\item[c)] $J$ is an unmixed matroidal ideal of degree 2.
\end{enumerate}

There are examples illustrating the significance of each of the
items in the above characterization  and showing that none of them
can be removed, see Examples \ref{exam1} and \ref{exam2}.

\section{preliminaries}

Throughout this paper $S=k[x_1,\ldots,x_n]$ is the polynomial ring
over a field $k$ with  the unique homogenous maximal ideal
$\fm=(x_1,\ldots,x_n)$.
 For a monomial ideal $I$ of $S$, the minimal set of monomial generators of $I$ is denoted by $\G(I)$ and
 $\supp(I):=\{x_i ; 1\leq i\leq n, x_i|u \mbox{ for some } u\in\G(I) \}$.
  We call the monomial ideal $I$ \emph{fully supported} if $\supp(I)=\{x_1,\ldots,x_n\}$.
  An ideal $I$ is said to be \emph{unmixed} if all associated prime
  ideals of $I$ have the same height and is called \emph{equidimensional} if all minimal prime ideals have the same height.

  A monomial ideal $I$ is called  a \emph{polymatroidal} ideal, if it is generated in a
  single degree with the exchange  property that  for any two
   elements $u,v\in \G(I)$ with  $\deg_{x_i}(u)> \deg_{x_i}(v)$, there exists an index $j$  with
   $\deg_{x_j}(u)< \deg_{x_j}(v)$ such that $x_j(u/x_i)\in \G(I)$.
It is easy to see that a monomial ideal $I$ is polymatroidal if and only if for all monomials
 $u,v\in\G(I)$ with $\deg_{x_i}(u)>\deg_{x_i}(v)$ for some $i$,  there  exists an integer $j$
such that $\deg_{x_j}(v)>\deg_{x_j}(u)$ and $x_j(u/x_i)\in I$. A squarefree polymatroidal ideal is called a \emph{matroidal} ideal.

Recall that any polymatroidal ideal $I$ has a linear resolution by \cite[Lemma 1.3]{HT} and \cite[Lemma 4.1]{CH}.
As a consequence the Castelnuovo-Mumford regularity of $I$ is equal to $d$,
where $I$ is generated in degree $d$ and we have the following  presentation for $I$ which we will use it frequently in our approach.

\begin{prop}\cite[Proposition 5]{HV}\label{poly}
For a   polymatroidal ideal $I\subset S$ with
$\Ass(S/I)\setminus\{\fm\}=\{\fp_1,\ldots,\fp_r\}$, there are
integers $a_i>0$ and $s\geq 0$ such that
$I=\fp_1^{a_1}\cap\cdots\cap\fp_r^{a_r}\cap\fm^s$ and $I$ is
generated in degree $s$, when $s>0$.
\end{prop}

The following observation shows that an unmixed  polymatroidal
ideal generated in degree 2, is not very far from a matroidal ideal.

\begin{lem}\label{poly2}
Let $I$ be a fully supported polymatroidal ideal of $S$, generated
in degree 2. If  $I$ is unmixed, then $I$ is a matroidal ideal or
$I=\fm^2$.
\end{lem}

\begin{proof}
If $|\Ass(S/I)|=1$, then the result is clear. Otherwise, let
$I=\fp_1^{a_1}\cap\cdots\cap\fp_r^{a_r}$ be the minimal primary
decomposition mentioned in Proposition \ref{poly}. Since
$\h(\fp_i)=\h(\fp_j)$ for all $i\neq j$, there exist
$x_i\in\fp_i\setminus\fp_j$ and $x_j\in\fp_j\setminus\fp_i$.
Therefore $x_i^{a'_i}x_j^{a'_j}|u$  for  $a'_i\geq a_i\geq 1$,
$a'_j\geq a_j\geq 1$ and some $u\in\G(I)$. Now, since $\deg(u)=2$,
we have that $a_i=a_j=1$ for all $i\neq j$ and so $I$ is a matroidal
ideal.
\end{proof}

The unmixed condition is necessary in the above lemma. For instance consider the equidimensional
polymatroidal ideal  $I=(x_1^2,x_1x_2,x_1x_3,x_2x_3)=(x_1,x_2)\cap(x_1,x_3)\cap(x_1,x_2,x_3)^2$.
The point in this example is that
$I$ contains a pure power of a variable $x_1$ but not any other powers $x_2^2$ or $x_3^2$. The following
result shows that it can not happen if the ideal is unmixed.

\begin{prop}
Let $I$ be an unmixed fully supported polymatroidal ideal of $S$,
generated in degree $d$. If $x_j^d\in I$ for some $1\leq j\leq n$,
then $I=\fm^d$.
\end{prop}
\begin{proof}
Let $I=\fp_1^{a_1}\cap\cdots\cap\fp_r^{a_r}$ be the minimal primary
decomposition. Since $x_j^d\in I$ for some $j$, we have
$x_j^d=x_j^{\max\{a_i ; i=1,\ldots,r\}}$.  For simplicity, assume that $d=a_1$. Since $I=\fp_1^{d}\cap\cdots\cap\fp_r^{a_r}$ is generated in degree
$d$ and $\fp_i\nsubseteq\fp_1$ for all $i=2,\ldots,r$, we get  $I=\fp_1^d=\fm^d$.
\end{proof}

The unmixed polymatroidal ideals which appear in the above
statement, powers of the maximal ideal $\fm$, are called Veronese
ideals. In other words
 the (\emph{squarefree}) \emph{Veronese ideal} of degree $d$ in the variables $x_{i_1},\ldots,x_{i_r}$ is
the ideal of $S$ which is generated by all (squarefree) monomials in
$x_{i_1},\ldots,x_{i_r}$ of degree $d$.

\begin{thm}\cite[Theorem 4.2]{HH}\label{CM}
A polymatroidal ideal $I$ is Cohen-Macaulay if and only if $I$ is a principal ideal,  a Veronese ideal or  a squarefree Veronese ideal.
\end{thm}

As a generalization of  Veronese ideals, if the ideal $I$ is
generated by all monomials $u$ of degree $d$ such that
$\deg_{x_i}(u)\leq a_i$ for some integers $a_i\geq 0$, the ideal $I$
is denoted by $I_{d;a_1,\ldots,a_n}$ and is called an ideal of
\emph{Veronese type}. Ideals of Veronese type are obviously
polymatroidal. If $I$ is an ideal of  Veronese type, then
$\Min(S/I)=\Ass(S/I)$ if and only if $I$ is unmixed if and only if
$I$ is Cohen-Macaulay, see \cite[Theorem 3.4]{V}.

Let $\fp$ be a monomial prime ideal of $S$. Then
$\fp=\fp_{\{i_1,\ldots,i_t\}}$ where
$\{i_1,\ldots,i_t\}=[n]\setminus\{i; x_i\in\supp(\fp)\}$ and
$IS_\fp=JS_\fp$, where $J$ is the monomial ideal  obtained from $I$
by
 the substitution $x_i\mapsto 1$ for all $i=i_1,\ldots,i_t$.
The ideal  $J$ is called  the \emph{monomial localization} of $I$ with respect to $\fp$ and is denoted by $I(\fp)$.
The following easy observation is a crucial point in using monomial localization as an effective tool.
\begin{rem}\label{Q}
Let $I=\cap^r_{i=1}Q_i$ be a primary decomposition of a monomial
ideal $I$.
\begin{enumerate}
\item[a)] $I(\fp_{\{1,\ldots,t\}})=\underset{i\in T}{\cap} Q_i$ where $T=\{i; 1\leq i\leq r, \supp(Q_i)\cap\{x_1,\ldots,x_t\}=\emptyset\}$.
\item[b)] If $I$ is generated in single degree $d$ and $I(\fp_{\{i\}})$ is generated in  single degree  $d_i$, then $d_i=d-a_i$ where
$a_i=\max\{\deg_{x_i}(u) ; u\in\G(I)\}$ and  \[\G(I(\fp_{\{i\}}))=\{\frac{u}{x_i^{a_i}} ; u\in\G(I) \mbox{ and } x_i^{a_i}|u\}.\]
\end{enumerate}
\end{rem}

\section{polymatroidal ideals connected in codimension one}

In this section we study the Cohen-Macaulay property of
polymatroidal ideals from topological point of view. Let $I$ be a
monomial ideal of $S$  and  consider the Zarisky topology on
$\Spec(S/I)$. Recall that the closed subsets in this topology are
the sets $\V(J)=\{\fq ; \fq\in\Spec(S) \mbox{ and }
J\subseteq\fq\}$, where $J\supseteq I$ is an ideal of $S$. The
irreducible components of $\Spec(S/I)$ are the closed sets
$\V(\fp)$, where $\fp$ is a minimal prime ideal of $I$.
$\Spec(S/I)$ with this topology is a connected space. The ideal $I$
is called \emph{connected in codimension one}, if $\Spec(S/I)$
remains connected after removing closed subsets with codimension
bigger than one \cite{Ha}. Since the codimension of $\V(\fp)$  is
equal to $\h(\fp)-\h(I)$ for all prime ideals $\fp\supseteq I$,  we
have the following definition by \cite[Proposition 1.1]{Ha}.

\begin{defn}\em{
A monomial ideal $I\subset S$ with height $h$, is connected in
codimension one, if for any pair of distinct prime ideals $\fp,
\fq\in\Min(S/I)$ there exists a sequence of minimal prime ideals
$\fp=\fp_1,\ldots,\fp_r=\fq$ such that $|\G(\fp_i+\fp_{i+1})|=h+1$,
for all $1\leq i\leq r-1$.}
\end{defn}

\begin{rem}\label{equ}\em{
By the above definition it is clear that a monomial ideal connected
in codimension one  is equidimensional and so
$|\G(\fp_i)\cap\G(\fp_{i+1})|=h-1$, for all $1\leq i\leq r-1$. Since
for a squarefree monomial ideal $I$, all associated prime ideals are
minimal, being equidimensional is equivalent to being unmixed. Thus if a squarefree monomial ideal $I$ is connected in codimension
one, then $I$ is unmixed.  }
\end{rem}

\begin{rem}\label{hart}
\em{In the context of Hartshorne \cite{Ha}, an   ideal $I\subset S$
is called locally connected in codimension one if all localizations
$I_\fp$  are connected in codimension one where $\fp\in\V(I)$. If
$I$ is a monomial ideal  and   $I_\fm$ is  connected in codimension
one, then $I$ is connected in codimension one.}
\end{rem}

From combinatorial point of view,  a pure simplicial complex
$\Delta$ is said to be {strongly connected} or {connected in
codimension one}, if for any two facets $F$ and $G$, there is a
sequence of facets $F=F_1, F_2,\ldots,F_r=G$ such that $\dim(F_i\cap
F_{i+1})=\dim\Delta-1$ or equivalently $\dim(F_i\cup
F_{i+1})=\dim\Delta+1$, for each $1\leq i\leq r-1$. A squarefree
monomial ideal is connected in codimension one, if it is the
Stanley-Reisner ideal of a strongly connected simplicial complex.

\begin{rem}\label{def}
\em{ Let $I$ be a Cohen-Macaulay monomial ideal. Then $I$ is
connected in codimension one, by \cite[Corollary 2.4]{Ha} and Remark
\ref{hart}. Another way to see this fact is observing that
$\sqrt{I}$ is also Cohen-Macaulay by \cite[Theorem 2.6]{HTT}. So
according to \cite[Lemma 9.1.12]{HH1}, $I$ is connected in
codimension one, since $\Min(S/I)=\Min(S/\sqrt{I})$.}
\end{rem}

Obviously an unmixed principal ideal is connected in codimension
one. As an easy way to construct a monomial ideal connected in
codimension one, we may  consider $I$ as the intersection of all
prime ideals generated by $h=\h(I)$ variables. It is indeed the
squarefree Veronese ideal generated in degree $d=n-h+1$
\cite[Theorem 3.4]{C}. From another point of view,
 $I$ is Cohen-Macaulay by Theorem \ref{CM} and  hence
  $I$ is
connected in codimension one by the above remark.

In Theorem \ref{pc}, we show that all matroidal  ideals connected in
codimension one, are precisely the squarefree Veronese ideals. As a
key point of our proof, we need the following simple
characterization which in  the case  $t=2$ is also proved by a
different method in  \cite[Lemma 2.3]{C}.

\begin{lem}\label{loc}
Let $I$ be a matroidal ideal  and
$T=\{x_1,\ldots,x_t\}\subseteq\supp(I)$. If for any $t-1$ elements
$x_{j_1},\ldots,x_{j_{t-1}}$ of $T$,  $x_{j_1}\cdots x_{j_{t-1}}|u$
for some $u\in\G(I)$, then the following statements are equivalent.

\emph{a)} $x_1\cdots x_t\nmid u$ for all $u\in\G(I)$.

\emph{b)}
$I(\fp_{\{1,\ldots,{t}\}})=I(\fp_{\{{j_1},\ldots,{j_{t-1}}\}})$ for
all $\{x_{j_1},\ldots,x_{j_{t-1}}\}\subseteq T$.

\emph{c)} $|\supp(\fp)\cap\{x_1,\ldots,x_t\}|\neq 1$ for all
$\fp\in\Ass(S/I)$.
\end{lem}

\begin{proof}
$(a)\Rightarrow(b)$: By \cite[Corollary 3.2]{HRV} any monomial
localization of $I$ is again matroidal and so it is generated in a
single degree. Since $x_{j_1}\cdots x_{j_{t-1}}|u$ for some
$u\in\G(I)$, we have $I_j=I(\fp_{\{j_1,\ldots,j_{t-1}\}})$ is
generated in degree $d-t+1$ where $d$ is the degree of the
generators of $I$. Indeed,
\[\G(I_j)=\{\frac{u}{x_{j_1}\cdots x_{j_{t-1}}} ;  u\in\G(I) \mbox { and } x_{j_1}\cdots x_{j_{t-1}}|u \}.\]
On the other hand $x_1\cdots x_t\nmid u$ for all $u\in\G(I)$, therefore $x\notin\supp(I_j)$ for $x\in T\setminus \{x_{j_1},\ldots,x_{j_{t-1}}\}$, and it follows (b).

$(b)\Rightarrow(c)$: Assume that $x_i\in\fp$ for some $\fp\in\Ass(S/I)$ and $1\leq i\leq t$.
Therefore $\fp\notin\Ass(S/I(\fp_{\{1,\ldots,{t}\}}))=\Ass(S/I(\fp_{\{1,\ldots,{i-1},i+1,\ldots,t\}}))$,
that is $x_j\in\fp$ for some $1\leq j\neq i\leq t$.

$(c)\Rightarrow(a)$: Assume  that $x_1\cdots x_t|u$ for some
$u\in\G(I)$. Then $x_t\in\supp(I(\fp_{\{1,\ldots,t-1\}}))$ and so
there exists a prime ideal $\fp\in\Ass(S/I(\fp_{\{1,\ldots,t-1\}}))$
such that $x_t\in\fp$. Now (c) implies that $x_i\in\fp$ for some
$1\leq i\leq t-1$ which is a contradiction.
\end{proof}

Now we are able to classify all  matroidal ideals  connected in
codimension one.

\begin{thm}\label{pc}
Let $I$ be a monomial ideal. Then $I$ is a matroidal ideal connected
in codimension one if and only if $I$ is a squarefree Veronese
ideal.
\end{thm}

\begin{proof}
If $I$ is squarefree Veronese ideal, then $I$ is matroidal ideal and
connected in codimension one by the explanation  after Remark
\ref{def}. Assume that $I$ is a matroidal ideal generated in degree
$d$ and is connected in codimension one. We use induction on $i$,
$1\leq i\leq d$ to  show that for any  set
$\{x_{j_1},\ldots,x_{j_i}\}\subseteq\supp(I)$, there exists $u\in
\G(I)$ such that $x_{j_1}\cdots x_{j_i}|u$.

Our claim is trivial for $i=1$. Assume that it's true for $i=t-1$
and assume contrary that $t\leq d$ and
$\{x_{1},\ldots,x_{t}\}\subseteq\supp(I)$ and  $x_{1}\cdots
x_{t}\nmid u$ for all $u\in\G(I)$. By induction assumption, for any
subset $\{x_{r_1},\ldots, x_{r_{t-1}}\}$ of $t-1$ elements of
$\{x_1,\ldots,x_t\}$, $x_{r_1}\cdots x_{r_{t-1}}|u$ for some
$u\in\G(I)$.   Note that by Lemma \ref{loc},
$I(\fp_{\{1,\ldots,t\}})=I(\fp_{\{1,\ldots,t-1\}})$ and
$I(\fp_{\{1,\ldots,t-1\}})\neq S$, since $t-1<d$ and $x_1\cdots
x_{t-1}|u$ for some $u\in\G(I)$. Hence there exists
$\fq\in\Ass(S/I)$ such that $\{x_1,\ldots,x_t\}\cap
\supp(\fq)=\emptyset$. Let $\fp\in\Ass(S/I)$ with $x_1\in\fp$. Since
$I$ is connected in codimension one by Remark \ref{equ}, there
exists a chain $\fp=\fp_1,\ldots,\fp_r=\fq$ of associated prime
ideals of $I$ such that $|\G(\fp_s)\cap \G(\fp_{s+1})|=\h(I)-1$ for all
$1\leq s\leq r-1$. As $x_1\in\fp_1$, by Lemma \ref{loc} we have $|\supp(\fp_1)\cap\{x_1,\ldots,x_t\}|\geq 2$. On the other hand, $|\G(\fp_1)\cap \G(\fp_{2})|=\h(I)-1$. Therefore
$|\supp(\fp_2)\cap\{x_1,\ldots,x_t\}|\geq 2$. Continuing in this way, we get $|\supp(\fq)\cap\{x_1,\ldots,x_t\}|\geq 2$, a contradiction.
\end{proof}

By Remark \ref{equ}, matroidal ideals connected in codimension one
are unmixed. The following example shows that this is not true for
polymatroidal ideals which are connected in codimension one.

\begin{ex}\label{exam-c2}\em{
 The ideal $I=(x_1^3,x_1^2x_2,x_1^2x_3,x_1x_2x_3,x_1x_2^2)=(x_1)\cap(x_1,x_2)^2\cap(x_1,x_2,x_3)^3$
 is polymatroidal which is clearly connected in codimension one,  but
it is not unmixed. }
\end{ex}

In our main result Theorem \ref{gc},  we show that an unmixed
 polymatroidal ideal is connected in codimension one if and only if
it is Cohen-Macaulay. We will use the following easy lemma, in our
proof.

\begin{lem}\label{h}
Let $I\subset k[x_1,\ldots,x_n]$ be an unmixed fully supported
polymatroidal ideal with $\h(I)>1$. If $I$ is not squarefree, then
$\h(I)\neq n-1$.
\end{lem}

\begin{proof}
Assume  that $\h(I)=n-1$. Then  $S/I$ is not Cohen-Macaulay by
Theorem \ref{CM} and $\dim(S/I)=1$. Therefore   $\depth(S/I)=0$ and
so $\fm\in\Ass(S/I)$ which contradicts $\h(I)=n-1$ and the
assumption that $I$ is unmixed.
\end{proof}

 Now, we present the main result of this section, which states that

\begin{thm}\label{gc}
Let $I$ be an unmixed  polymatroidal ideal. Then  $I$ is connected
in codimension one if and only if  $I$ is Cohen-Macaulay.
\end{thm}

\begin{proof}
If $I$ is Cohen-Macaulay, then $I$ is connected in codimension one
by Remark \ref{def}. Now let
$I=\fp_1^{a_1}\cap\cdots\cap\fp_r^{a_r}$ be connected in codimension
one. We may assume that $I$ is  fully supported with $\h(I)>1$  and
 is not squarefree,  by Theorem \ref{CM} and Theorem \ref{pc}.
Therefore to prove that $I$ is Cohen-Macaulay, according to Theorem
\ref{CM}, we must show that $r=1$. We use induction on $d$, which is
the common degree of monomial generators of $I$. For $d=2$, the
result follows by Lemma \ref{poly2}. Let $d>2$ and $a_i>1$ for some
$1\leq i\leq r$ and assume contrary that $r>1$. Since $I$ is
connected in codimension one by Remark \ref{equ}, there exist $1\leq
j\neq i\leq r$ such that $|\G(\fp_i)\cap\G(\fp_j)|=\h(I)-1$. Note
that $\h(I)\neq n-1$ by Lemma \ref{h} and so  $\fp:=\fp_i+\fp_j\neq
\fm$. On the other hand $\supp(I(\fp))=\G(\fp_i)\cup\{x\}$ for some
variable $x\in\fp_j\setminus\fp_i$.
 Now let $\fq\in\Ass(S/I(\fp))$ and $\fq\neq\fp$. Then  $\G(\fq)\subseteq\G(\fp_i)\cup\{x\}$. Since $\h(\fp_i)=\h(\fq)$,
 then $\fq=(\G(\fp_i)\setminus\{y_\fq\},x)$ for some variable $y_\fq$. Hence $I(\fp)$ is a polymatroidal ideal connected in codimension
 one
  which is generated in degree less than $d$.  Now, induction assumption implies  $|\Ass(S/I(\fp))|=1$ which is a contradiction.
\end{proof}

\begin{cor}\label{d=2}
Let $I\subset S$ be an unmixed fully supported polymatroidal ideal
and connected in codimension one. Then $\supp(I(\fp_{\{i\}}))$ is
either an empty set or is equal to
$\{x_1,\ldots,x_n\}\setminus\{x_i\}$ for each $i=1,\ldots,n$.
\end{cor}

\begin{proof}
If $I$ is a squarefree Veronese ideal in variables $x_1,\ldots,x_n$,
then  for all $1\leq i,j\leq n$, $x_ix_j|u$ for some $u\in\G(I)$.
Hence the result is clear by Theorem \ref{gc} and Theorem \ref{CM}.
\end{proof}

\begin{cor}\label{co}
Let $I$ be a polymatroidal ideal. Then $I$ satisfies  Serre's
condition $(S_n)$ for some $n\geq 2$ if and only if $I$ is
Cohen-Macaulay.
\end{cor}

\begin{proof}
Assume that $I$ satisfies  Serre's condition $(S_n)$ for some
$n\geq 2$. Then $I$ is connected in codimension one by
\cite[Corollary 2.4]{Ha} and Remark \ref{hart}. Hence  $I$ is
equidimensional by Remark \ref{equ} and so $I$ is unmixed since it
is $(S_1)$. Now the result follows by Theorem \ref{gc}.
\end{proof}

\section{Generalized Cohen-Macaulay polymatroidal ideals}

A finitely generated module $M$ over a local ring $(R,\fn)$ is
called \emph{generalized Cohen-Macaulay},  whenever each local
cohomology module $\H^i_\fn(M)$ has  finite length for all $i < \dim
M$. It is known that if $M$ is generalized Cohen-Macaulay, then
$M_\fp$ is Cohen-Macaulay for all prime ideals
$\fp\in\Spec(R)\setminus\{\fn\}$ and the converse holds if $R$ is
universally catenary and all its formal fibres are Cohen-Macaulay
\cite[Exercises 9.5.7 and 9.6.8]{BS}. In the following we consider
graded generalized Cohen-Macaulay modules over the $^*$local graded
polynomial ring $(S,\fm)$. We call an ideal $I$  generalized
Cohen-Macaulay whenever  the $i$th cohomology module $\H^i_\fm(S/I)$
is  of finite length for all $i<\dim(S/I)$.

\begin{lem}\label{eq}
The following statements are equivalent for a monomial ideal $I$.
\begin{enumerate}
\item[a)] $I$ is generalized Cohen-Macaulay.
\item[b)] $I$ is equidimensional and $I(\fp)$ is Cohen-Macaulay for all monomial prime ideals $\fp\neq
\fm$.
\end{enumerate}
\end{lem}
\begin{proof}
   Note that homogeneous prime ideals of $S$ in  multigraded structure are precisely monomial prime ideals. Thus all minimal elements
   of the non Cohen-Macaulay locus of $S/I$ are monomial by \cite[Corollary 3.7]{PS}. Now the result follows by  \cite[Exercise 9.5.7]{BS}.
\end{proof}

\begin{prop}\label{2}
Let $I$ be a polymatroidal ideal generated in degree 2. Then the following statements are equivalent:
\begin{enumerate}
\item[a)] $I$ is equidimensional.
\item[b)] $I$ is generalized Cohen-Macaulay.
\end{enumerate}
\end{prop}

\begin{proof}
Note  that $I(\fp_{\{i\}})$ is generated by indeterminates or is
equal to $S$, and so it is Cohen-Macaulay for all $i=1,\ldots,n$.
\end{proof}

The following example shows that the above result is not true for polymatroidal ideals generated in degree $d>2$.

\begin{ex}\label{1.4}
The ideal
$I=(uxy,uyz,uzw,uxw,xyz,wxz)=(x,u)\cap(x,z)\cap(y,w)\cap(z,u)$ is an
unmixed matroidal ideal. But it is not generalized Cohen-Macaulay,
since $I(\fp_{\{x\}})=(uy,uw,yz,wz)$ is not Cohen-Macaulay by
Theorem \ref{CM}.
\end{ex}

In the above example, $I$ is an unmixed ideal which is not generalized Cohen-Macaulay. As a main result of this section
 we will show in Theorem \ref{mat} that the generalized Cohen-Macaulay matroidal ideals generated in degree $d>2$ are
 precisely  Cohen-Macaulay matroidal ideals.
In order to prove this, we need the following result which is
interesting in its own.

\begin{prop}\label{mat1}
Let $I\subset k[x_1,\ldots,x_n]$ be a fully supported matroidal
ideal generated in degree $d>2$. If $I$ is generalized
Cohen-Macaulay, then
$\supp(I(\fp_{\{i\}}))=\{x_1,\ldots,x_n\}\setminus\{x_i\}$ for each
$i=1,\ldots,n$.
\end{prop}

\begin{proof}
We show,  for convenience, that
$\supp(I(\fp_{\{1\}}))=\{x_2,\ldots,x_n\}$. Let $I(\fp_{\{1\}})$ be
 fully supported in $K[x_{t+1},\ldots,x_n]$ for some $t\geq 1$. It is enough to show
 that $t=1$. Since $I(\fp_{\{1\}})$ is a squarefree Veronese ideal in variables
 $x_{t+1},\ldots,x_n$
 of degree $d-1$, where $d$ is the common degree of generators of
 $I$,
  it follows that
\begin{equation}\label{dag}
h=\h(I)=(n-t)-(d-1)+1=n-t-d+2.
\end{equation}

Since $d> 2$, we have that for each $j=t+2,\ldots,n$,
 there exists $u_j\in\G(I)$ such that $x_{t+1}x_j|u_j$ and so
\begin{equation}\label{*}
 \{x_{t+2},\ldots,x_n\}\subseteq \supp(I(\fp_{\{t+1\}})).
 \end{equation}
  On the other  hand since $\supp(I(\fp_{\{1\}}))=\{x_{t+1},\ldots,x_n\}$, it follows
 that $x_1x_j\nmid u$ for each $j=2,\ldots,t$ and any $u\in\G(I)$. So Lemma \ref{loc} implies that $I(\fp_{\{1\}})=I(\fp_{\{j\}})$ for
 $j=2,\ldots,t$. Now again since
 $\supp(I(\fp_{\{j\}}))=\{x_{t+1},\ldots,n\}$ for $j=1,\ldots,t$, we
 have  that for each $j=1,\ldots,t$,
 there exists $u_j\in\G(I)$ such that $x_jx_{t+1}|u_j$. Thus
 \begin{equation}\label{**}
\{x_1,\ldots,x_t\}\subseteq \supp(I(\fp_{\{t+1\}})).
\end{equation}
  Hence by (\ref{*}) and (\ref{**}), we have that $\supp(I(\fp_{\{t+1\}}))=\{x_1,\ldots,x_n\}\setminus\{x_{t+1}\}$.
 Therefore since $I(\fp_{\{t+1\}})$ is a squarefree Veronese ideal,
 it follows that  $h=(n-1)-(d-1)+1=n-d+1$. Hence
 from (\ref{dag}), $n-t-d+2=n-d+1$. So $t=1$.
\end{proof}

\begin{thm}\label{mat}
Let $I$ be a matroidal ideal generated in degree $d>2$. Then $I$ is
generalized Cohen-Macaulay if and only if $I$ is Cohen-Macaulay.
\end{thm}

\begin{proof}
By Proposition \ref{mat1}, $I(\fp_{\{i\}})$ is a squarefree Veronese ideal in the variables
$\{x_1,\ldots,x_n\}\setminus\{x_i\}$, for all $1\leq i\leq n$. Now, since $I=\sum^n_{i=1}x_iI(\fp_{\{i\}})$, the result is clear.
\end{proof}

The following lemma will be used in the classification of generalized Cohen-Macaulay polymatroidal ideals in Theorem \ref{th}.
\begin{lem}\label{exc}
Let $I=J\cap \fm^d$ be a polymatroidal ideal generated in degree $d$ where $J$ is a squarefree monomial ideal.
If $\deg(u)>1$ for all $u\in \G(J)$, then $J$ is a matroidal ideal.
\end{lem}
\begin{proof}
Let $u,v\in \G(J)$ such that $x_i|u$ and $x_i\nmid v$. Then $x_l|v$
and $x_l\nmid u$ for some $l\neq i$. By assumption  there exists
$h\neq i$ such that $x_h|u$. Now $u'=x_h^{d-s}u$ and $v'=x_l^{d-r}v$
belong to $\G(I)$, where $r=\deg(v)$ and $s=\deg(u)$. Since
$\deg_{x_i}(u')>\deg_{x_i}(v')$ and $I$ is polymatroidal ideal,
there exists $1\leq j\neq i\leq n$ such that
$\deg_{x_j}(u')<\deg_{x_j}(v')$ and $x_ju'/x_i\in \G(I)$. Hence
$x_ju'/x_i\in J$. Note that $J$ is squarefree, $x_h|u$ and $h\neq
i$, thus $x_ju/x_i\in J$ and also $\deg_{x_j}(u)<\deg_{x_j}(v)$.
\end{proof}

\begin{lem}\label{cap-prod}
Let $I=J\cap\fm^d$ be a monomial ideal generated in degree $d$ where
$J$ is a monomial ideal generated in degree $t\leq d$. Then
$I=J\fm^{d-t}$.
\end{lem}

\begin{proof}
It is clear that $J\fm^{d-t}\subseteq I$. Now let $u\in\G(I)$, so
there exists $v\in\G(J)$ such that $v|u$. So since $\deg(v)=t\leq
d=\deg(u)$, there exists a monomial $w$ of degree $d-t$ such that
$u=vw$. Hence $u\in J\fm^{d-t}$.
\end{proof}

\begin{thm}\label{th}
Let $I=J\cap\fm^s$ be a fully supported monomial ideal in
$S=K[x_1,\ldots,x_n]$ and generated in degree $d$, where
$s\in\{0,d\}$. Then $I$ is  a generalized Cohen-Macaulay polymatroidal
ideal if and only if one of the following statements holds true:
\begin{enumerate}
\item[a)] $J$ is a Cohen-Macaulay polymatroidal   ideal i.e. $J$ is either a principal ideal, a Veronese ideal, or a squarefree Veronese ideal.
\item[b)] $J=\fp_1^{a_1}\cap\cdots\cap\fp_r^{a_r}$ is equidimensional and $\fp_i+\fp_j=\fm$ for all $i\neq j$.
\item[c)] $J$ is an unmixed matroidal ideal of degree 2.
\end{enumerate}
\end{thm}

\begin{proof}
By Lemma \ref{cap-prod}, each of statements (a) or (c) implies
that  $I$ is polymatroidal. Since $I$ is generated in a single
degree, from  statement (b)  follows that $I$ is polymatroidal by
\cite[Theorem 3.1]{FV}.

Whenever (a) holds, $I$ is equidimensional, since $J$ is unmixed. On the other hand for all monomial prime $\fp\neq \fm$, $I(\fp)=J(\fp)$ is Cohen-Macaulay.

Assume that (b) holds and let $\fq\in\V(I)\setminus\{\fm\}$ be a monomial prime ideal. Since $\fp_i+\fp_j=\fm$ for all $i\neq j$
and $\fq\neq \fm$, we get $I(\fq)=\fp_k^{a_k}$ for some $k$, $1\leq k\leq r$ or  $I(\fq)=S$.

Assume that (c) holds. By  Proposition \ref{2}, $J$ is generalized
Cohen-Macaulay. Thus for all monomial prime $\fp\neq \fm$,
$I(\fp)=J(\fp)$ is Cohen-Macaulay.

Conversely, assume that $I$ is a generalized
Cohen-Macaulay  polymatroidal   ideal, (a) and (b) don't hold. Note that
$J:=\fp_1^{a_1}\cap\ldots\cap\fp_r^{a_r}$ is an unmixed ideal. Since
(b) doesn't hold, let for convenience  $\fq=\fp_1+\fp_2\neq\fm$. Then $I(\fq)=\fp_1^{a_1}\cap\fp_2^{a_2}\cap\cdots\cap\fp_t^{a_t}$ is Cohen-Macaulay for some $2\leq t\leq r$.
Since the unmixed ideal $J$ is not principal,  $J(\fq)=I(\fq)$ is not principal.  Now, Theorem \ref{CM}
  implies that  $I(\fq)$ is squarefree Veronese ideal. Therefore $a_1=\cdots=a_t=1$ and
  so $I=\fp_1\cap\cdots\cap\fp_t\cap\fp_{t+1}^{a_{t+1}}\cap\cdots\cap\fp_r^{a_r}\cap\fm^s$.

We claim that $a_{t+1}=\cdots=a_r=1$. Otherwise, there exists
$t+1\leq i\leq r$ such that $a_i\neq 1$. Since $\fp_1\nsubseteq
\fp_{i}$, there exists a variable $x_l\in\fp_1\setminus\fp_{i}$.
Note that  $x_l\notin \cap^{t}_{j=1}\supp(\fp_j)$, since $I(\fq)$ is
generated in a single degree and is not a prime ideal. Let
$x_l\notin \fp_j$ for $1\leq j\leq t$. Then
$I(\fp_{\{l\}})=\fp_j\cap\fp_{i}^{a_{i}}\cap\fq'$ which is not
Cohen-Macaulay.  This contradiction implies our claim that $I=J\cap
\fm^s$ where $J$ is a squarefree monomial ideal. Since $I(\fq)$ is
squarefree  Veronese ideal of height greater than one, $J$ does not
contain any variables since $J(\fq)=I(\fq)$. Now,   the result
follows by Lemma \ref{exc} and Theorem \ref{mat}, since $J$ is not
Cohen-Macaulay.
\end{proof}

The following  examples  show that in the  above characterization,
none of the items (a), (b) or (c)  can be removed.

\begin{ex}\label{exam1} \em{
(i) The ideal $I=(x_1x_2^3,x_1^2x_2^2)=(x_1)\cap(x_2^2)\cap(x_1,x_2)^4$ is polymatroidal which satisfies  (a) and (b), but (c) doesn't hold for it.

(ii) The ideal $I=(x_1^2x_2,x_1x_2^2,x_1x_2x_3)=(x_1)\cap(x_2)\cap(x_1,x_2,x_3)^3$ is polymatroidal
which satisfies  (a) and (c), but (b) doesn't hold for it.

(iii) The ideal  $I=(x_1,x_2,x_3,x_4)\cap(x_3,x_4,x_5,x_6)\cap(x_1,x_2,x_5,x_6)$ constructed in \cite{HH},
 is matroidal ideal which satisfies  (b) and (c), but (a) doesn't hold for it. }
\end{ex}

\begin{ex}\label{exam2}\em{
 The ideal $I=(x_1,x_2)\cap(x_2,x_3)^2\cap(x_1,x_2,x_3)^3$ is polymatroidal by \cite[Theorem 3.1]{FV} and
 generalized Cohen-Macaulay  by Theorem \ref{th} satisfying condition (b), but $J=(x_1,x_2)\cap(x_2,x_3)^2$ is not even generated in a single degree.}
\end{ex}

Note that the above example is connected in codimension one. There exist polymatroidal ideals connected in codimension one, which are
not  generalized Cohen-Macaulay, see  Example \ref{exam-c2}. In this
example the localization $I(\fp_{\{3\}})=(x_1)\cap(x_1,x_2)^2$ is
not Cohen-Macaulay.

Polymatroidal ideals  which satisfy  condition (c) of Theorem
\ref{th}, can be specified by  the following lemma.

\begin{lem}\label{akhar}
Let $I$ be a fully supported monomial ideal of degree 2. Then $I$ is
polymatroidal if and only if $\fp_i+\fp_j=\fm$ for $i\neq j$ and all
$\fp_i\in\Ass(S/I)$.
\end{lem}

\begin{proof}
Let $\fp_i+\fp_j=\fm$ for $i\neq j$ and all $\fp_i\in\Ass(S/I)$.
Since $I$ is generated in  a single degree, it follows  by
\cite[Theorem 3.1]{FV} that $I$ is polymatroidal. Conversely, Let
$I=\fp_1^{a_1}\cap\fp_2^{a_2}\cap\cdots\cap\fp_t^{a_t}$ be
polymatroidal and $\fq=\fp_i+\fp_j\neq\fm$ for some $i\neq j$. Then
$I(\fq)=\fp_i^{a_i}\cap\fp_j^{a_j}\cap\fq'$ for some monomial ideal
$\fq'$. Since $I$ is generated in degree 2 it follows that the ideal
$I(\fq)$ is a monomial prime ideal or is equal to $S$, which is a
contradiction.
\end{proof}

By the above lemma, in the case (c) of Theorem \ref{th}, for any
pair of distinct prime ideals $\fp,\fq\in\Ass(S/J)$ we have
$\G(\fp+\fq)=\supp(J)$ and $\supp(J)$ is not necessarily equal to
the set of all variables. But in the case (b), the same condition
holds with the distinctive  point that $J$ is fully supported in
$\supp(I)$, see Example \ref{exam1}(ii).

\section*{Acknowledgments}
The authors would like to thank Hossein Sabzrou for fruitful
discussions and useful comments regarding this paper.


\end{document}